\documentclass[10pt,reqno]{amsart}
\usepackage{amssymb,amsthm,amsfonts,amsmath,amscd,amssymb,epsfig}
\usepackage[all]{xy}

\usepackage{enumerate,array,hyperref, graphicx,color,amsthm,
bbm,mathrsfs, a4wide}

\theoremstyle{plain}
\newtheorem{theorem}{Theorem}[section]

\newtheorem{proposition}[theorem]{Proposition}

\newtheorem{lemma}[theorem]{Lemma}

\theoremstyle{remark}

\theoremstyle{definition}

\newcommand{\la}{\langle}
\newcommand{\ra}{\rangle}

\newcommand{\R}{{\mathbb{R}}}
\newcommand{\C}{{\mathbb{C}}}

\newcommand{\comment}[1]{}

\begin{document}

\title[The Cartan geometry of the rotating Kepler problem]
{The Cartan geometry of the rotating Kepler problem}
\author{Kai Cieliebak}
\author{Urs Frauenfelder}
\author{Otto van Koert}
\address{
    Kai Cieliebak\\
    Ludwig-Maximilians-Universit\"at, 80333 M\"unchen, Germany}
    \email{kai@math.lmu.de}
\address{
    Urs Frauenfelder\\
    Department of Mathematics and Research Institute of Mathematics\\
    Seoul National University}
\email{frauenf@snu.ac.kr}
\address{
    Otto van Koert\\
    Department of Mathematics and Research Institute of Mathematics\\
    Seoul National University}
\email{okoert@snu.ac.kr} \keywords{Finsler geometry, Flag curvature,
3-body problem}

\begin{abstract}
We investigate the Cartan and Finsler geometry of the rotating Kepler problem, a limit case of the restricted three body problem that arises if the mass of the one of the primaries goes to zero.
We show that the Hamiltonian for the rotating Kepler problem can be regarded as the Legendre transform of a certain family of Finsler metrics on the two-sphere.
For very negative energy levels, these Finsler metrics are close to the round metric, and the associated flag curvature is hence positive.
On the other hand, we show that the flag curvature can become negative once the energy level becomes sufficiently high. 
\end{abstract}

\maketitle

\section{Introduction}

The Kepler problem in a rotating frame arises as the limit of the
restricted three body problem when the mass of one of the primaries
goes to zero. Energy hypersurfaces of the Kepler problem as well as
the restricted three body problem are noncompact. If one focuses on
the bounded components of the energy hypersurfaces, i.e.~one excludes
the possibility of escape to infinity, then the only obstruction to
compactness comes from collisions. However, it is well known in
celestial mechanics that two body collisions can always be
regularized. A particularly nice way to regularize the Kepler problem
in an inertial frame was discovered by Moser in \cite{M}. By
interchanging the roles of position and momenta Moser shows that the
regularized planar Kepler problem is equivalent to the geodesic flow
on the round two-sphere. Indeed, the point at infinity of the
two-sphere corresponds to collision where the momenta explode. In
\cite{AFKP} it was shown that for energies not too large the compact
components of the regularized energy hypersurfaces for the planar
restricted 3-body problem are starshaped. Here starshaped refers to
the new cotangent bundle structure obtained after interchanging the
roles of position and momenta; with respect to the old $(q,p)$-coordinates the Liouville vector field is given by
$q\frac{\partial}{\partial q}$ and not by $p\frac{\partial}{\partial
p}$.
The challenging new question is now if starshaped can be
replaced by the stronger condition of fiberwise convexity. If the energy
hypersurfaces are fiberwise convex we obtain a Finsler metric after Legendre
transformation and the integral curves of the restricted 3-body
problem can be interpreted as geodesics for the Finsler metric. The
relevance of this problem comes from its relation with the project
initiated by Helmut Hofer to construct finite energy foliations for
the restricted 3-body problem. Indeed, if the energy hypersurfaces
are fiberwise convex then the Conley-Zehnder indices of the closed
characteristics correspond to the Morse indices of the corresponding
geodesics and hence are nonnegative. Our first result is the following.
\\ \\
\textbf{Theorem~A:} \emph{The bounded components of the regularized
planar Kepler problem in a rotating frame are fiberwise convex.}
\\ \\
In particular, the regularized Hamiltonian of the rotating Kepler
problem is Legendre dual to a Finsler metric. The Legendre dual of a
Finsler space was recognized by Miron as a {\em Cartan space}, see
\cite{MHSS} and the literature cited therein. Therefore, Theorem~A
provides the rotating Kepler problem with a Cartan structure. The
sectional curvature of Riemannian geometry generalizes to the notion
of flag curvature in Cartan, respectively Finsler geometry. In
particular, for a Cartan surface the flag curvature is a function on
the unit tangent bundle of the surface. This differs from the
Riemannian case where the sectional curvature for a surface reduces
to the Gaussian curvature and hence is given by a function on the
surface itself. 
Since in our case the underlying surface is a sphere,
the flag curvature is a function on $\mathbb{R}P^3$ regarded as
the unit tangent bundle of the sphere. 
However, because the Kepler
problem is invariant under rotations, the flag curvature is
invariant under a circle action and quotienting out this circle
action we again get a function on a sphere. However, be aware that
this sphere does not coincide with the original one. The importance
of the flag curvature comes from its relation with the Jacobi
equation and hence with questions about indices, see \cite{BCS}. Our
next result is rather surprising since by Moser's result in an
inertial frame, the regularized Kepler problem is equivalent to the
geodesic flow of the round two sphere whose curvature is hence always positive.
\\ \\
\textbf{Theorem~B:} \emph{For the rotating Kepler problem the flag
curvature can become negative.}
\\ \\
Let us explain our motivation for studying this problem.
As mentioned above, the rotating Kepler problem is a limiting case of the restricted $3$-body problem.
Conley first proposed to use hyperbolicity in the restricted $3$-body problem in mission design, \cite{conley1,conley2,conley3}.
Suppose we have the following problem.
Given a satellite on an orbit around the earth, is it possible to move this satellite to an orbit around the moon using the gravitational forces of the earth, moon and maybe the sun without spending much fuel.
It was Belbruno who found the first realistic orbit using this idea, \cite{belbruno1}.
A variation of this orbit was succesfully put into practice to save the first Japanese lunar mission, \cite{belbruno-miller}.

Conley's idea takes advantage of a hyperbolic orbit around the first Lagrange point, also known as a Lyapunov orbit.
These Lyapunov orbits only appear for sufficiently high energy.

During a discussion with P.~Albers, E.~Belbruno, J.~Fish and H.~Hofer at the IAS in Princeton, the question arised whether (short) hyperbolic orbits can already arise below the first Lagrange point.
For very small masses of the moon or for very negative energies it is known that this phenomenon never happens.
Otherwise this question is completely open.

A possibility to locate the position of hyperbolic orbits in phase space would be to find regions of negative curvature of the Finsler metric.
Indeed, \cite{HP} contains a short computation which illustrates the relation between the flag curvature and the Conley-Zehnder index.
In particular, one can clearly see that if the flag curvature is negative along an entire orbit, then this orbit is hyperbolic.
We therefore hope that the methods studied in this paper have some applications for mission design.

We conclude the introduction by pointing out that one could also try to understand the Conley-Zehnder indices by investigating whether there exists a convex embedding of a level set of the restricted $3$-body problem into $\C^2$.
The latter question has been answered in the affirmative in \cite{AFFHK} provided the energy levels are sufficiently far below the first Lagrange value.

\section{A family of Minkowski metrics}

For $p=(p_1,p_2)\in\R^2$ and $C>0$ set $p^\perp=(p_2,-p_1)$ and
consider the function $H_p:\R^2-0\to\R$, 
$$
   H_p(q) := \la p^\perp,q\ra - \frac{1}{|q|}+2C.
$$
\begin{lemma}\label{lem:convex}
Suppose $|p|<C^2$. Then the curve  
$H_p^{-1}(0)$ has two connected components: an unbounded one in the region
$\{|q|>1/C\}$, and a bounded component $\Sigma_p$ in the region
$\{|q|<1/C\}$ which bounds a strictly convex domain containing the origin.
\end{lemma}

\begin{proof}
Given $q\neq 0$, let us determine all $\lambda\in\R-0$ such that
$H_p(\lambda^{-1}q)=0$, i.e.
$$
   |\lambda|\lambda - 2C|q|\lambda - |q|\la p^\perp,q\ra = 0. 
$$
Suppose without loss of generality that $\la
p^\perp,q\ra\geq 0$. Then the equation has the 3 solutions
$$
   \lambda_0 = C|q|\left(1 + \sqrt{1+\frac{\la p^\perp,q\ra}{|q|C^2}}
   \right)>0,\qquad
   \lambda_\pm = -C|q|\left(1 \pm \sqrt{1-\frac{\la p^\perp,q\ra}{|q|C^2}}
   \right)\leq 0.
$$
By hypothesis we have
$$
   0 \leq x= \frac{\la p^\perp,q\ra}{|q|C^2} < 1,
$$
and thus $1-\sqrt{1-x}\leq 1-(1-x)<1$. We see that
$|\lambda_0|,|\lambda_+|>C|q|$ and $|\lambda_-|<C|q|$. Thus
$\lambda_0^{-1}q$ and $\lambda_+^{-1}q$ belong to the bounded
component contained in $\{|q|<1/C\}$ and $\lambda_-^{-1}q$ to the
unbounded component contained in $\{|q|>1/C\}$.

For convexity we compute the gradient and Hessian
$$
   \nabla H_p(q) = \left(\begin{array}{cc}
      p_2+\frac{q_1}{|q|^3} \\
      -p_1 + \frac{q_2}{|q|^3}
   \end{array}\right),\qquad 
   \mathrm{Hess}H_p(q) = \frac{1}{|q|^5}\left(\begin{array}{cc}
      |q|^2-3q_1^2 &-3q_1q_2\\
      -3q_1q_2 & |q|^2-3q_2^2
   \end{array}\right).
$$
A tangent vector to $H_p^{-1}(0)$ is given by
$$
   v:=-\big(\nabla H_p(q)\big)^\perp=\left(\begin{array}{c}
      p_1-\frac{q_2}{|q|^3}\\
      p_2+\frac{q_1}{|q|^3}
   \end{array}\right).
$$
A short computation yields
\begin{align*}
   |q|^5\langle v,\mathrm{Hess} H_p(q) v\rangle
   &= \big(|q|^2-3q_1^2\big)\bigg(p_1-\frac{q_2}{|q|^3}\bigg)^2 -6q_1
   q_2\bigg(p_2+\frac{q_1}{|q|^3}\bigg)\bigg(p_1-\frac{q_2}{|q|^3}\bigg)
   \cr
   &\ \ \ + \big(|q|^2-3q_2^2\big)\bigg(p_2+\frac{q_1}{|q|^3}\bigg)^2 \cr
   &= |q|^{-2}\Bigl( |p|^2|q|^4 + 2|q|\la p^\perp,q\ra + 1 - 3|q|^2\la
   p,q\ra^2 \Bigr)
\end{align*}
To show positivity of the right hand side, we write $\la
p^\perp,q\ra=|p||q|\cos t$, $\la p,q\ra=|p||q|\sin t$. Then the right
hand side equals $|q|^{-2}f(t)$ with the periodic function 
$$
   f(t) := a^2 + 2a\cos t + 1 - 3a^2\sin^2 t, \qquad a := |p||q|^2. 
$$
The derivative $f'(t)=-2a\sin t(1+3a\cos t)$ vanishes iff $\sin t=0$
or $\cos t=-1/(3a)$. In the case $\sin t=0$ we have $f(t) = (1-a)^2 > 0$, 
since $0<a<1$ due to the conditions $|p|<C^2$ and $|q|<1/C$. In the
case $\cos t=-1/(3a)$ we have
$$
   f(t) = a^2 - \frac{2a}{3a} + 1 - 3a^2(1-\frac{1}{(3a)^2}) =
   2(\frac{1}{3}-a^2). 
$$
The equation $H_p(q)=0$ yields
$$
   \frac{1}{|q|}-2C = \la p^\perp,q\ra = |p||q|\cos t =
   -\frac{|p||q|}{3a} = -\frac{1}{3|q|},
$$
thus $|q|=2/(3C)$. Using this and $|p|<C^2$, we estimate
$$
   a = |p||q|^2 = \frac{4|p|}{9C^2} < \frac{4}{9}. 
$$
Hence $a^2<1/3$, which shows $f(t)>0$ and concludes the proof of the
lemma. 
\comment{
Alternatively, one can argue as follows.
The null spaces of the Hessian are spanned by the vectors
$$n_\pm=\left(\begin{array}{c}
q_1\pm\sqrt{2}q_2\\
q_2 \mp \sqrt{2}q_1
\end{array}\right)$$
and we compute
\begin{eqnarray*}
\langle -\nabla H_p(q),n_\pm\rangle &=&
p_1(q_2\mp\sqrt{2}q_1)-p_2(q_1\pm\sqrt{2}q_2)-\frac{1}{|q|}\\
&=&(p_1-q_2)(q_2\mp\sqrt{2}q_1)+q_2(q_2\mp\sqrt{2}q_1)
-(p_2+q_1)(q_1\pm\sqrt{2}q_2)+q_1(q_1\pm\sqrt{2}q_2)-\frac{1}{|q|}\\
&=&(p_1-q_2)(q_2\mp\sqrt{2}q_1)
-(p_2+q_1)(q_1\pm\sqrt{2}q_2)+|q|^2-\frac{1}{|q|}\\
&\leq&\sqrt{(p_1-q_2)^2+(p_2+q_1)^2}\sqrt{(q_2\mp\sqrt{2}q_1)^2+
(q_1\pm\sqrt{2}q_2)^2}+|q|^2-\frac{1}{|q|}\\
&\leq&\sqrt{3}|q|\sqrt{-3+\frac{2}{|q|}+|q|^2}+|q|^2-\frac{1}{|q|}\\
&<&0
\end{eqnarray*}
So the tangent vector $v$ is never in a null space, hence $\langle
v,\mathrm{Hess} H_p v\rangle$ has the same sign for all $p$ and all
$q$ on the bounded component $\Sigma_p$. Since for $p=0$ we have
$\langle v,\mathrm{Hess} H_0 v\rangle=|q|^{-7}>0$, this sign is always
positive.
}
\end{proof}

The convex curve $\Sigma_p$ endows $T_p^*\mathbb{R}^2$ with the
structure of a Cartan space with fundamental function
$F^*_p:\R^2\to\R_{\geq 0}$ in the sense of \cite{MHSS}: $F^*_p(q)$ is
the unique $\lambda>0$ such that $\lambda^{-1}q\in\Sigma_p$. This
$\lambda$ has been computed in the preceding proof to be
$$
   F^*_p(q) = C|q|\left(1 + \sqrt{1+\frac{\la p^\perp,q\ra}{|q|C^2}}
   \right).
$$
This Cartan structure is non-reversible (i.e.~$F_p^*(-q)\neq
F_p^*(q)$) for $p\neq 0$ and does not
appear to belong to one of the standard classes (Randers, Berwald,
...).

\section{Kepler problem in a rotating frame}
Consider the Kepler problem in a frame that is rotating around the
origin with angular velocity $a>0$. This system is again Hamiltonian
with Hamiltonian 
\begin{align*}
   H(p,q)
   &= a(p_2q_1-p_1q_2) - \frac{1}{|q|} + \frac{|p|^2}{2} \cr
   &=\frac{1}{2}\big((p_1-aq_2)^2+(p_2+aq_1)^2\big)+U(q),
\end{align*}
with the potential
$$
   U(q)=-\frac{1}{|q|}-\frac{a^2}{2}|q|^2.
$$
The function $U(r)=-\frac{1}{r}-\frac{a^2}{2}r^2$ attains its
maximum $-\frac{3}{2}a^{2/3}$ at the point $r=a^{-2/3}$. So the
energy level
$$
   \{H(p,q)=-c\},\qquad c>\frac{3}{2}a^{2/3}
$$
has an unbounded component and a bounded component $\Sigma$ on which
we have
$$
   |q|\leq a^{-2/3}.
$$
The function
$$
   H_p(q)=H(p,q)+c
$$
is of the type considered above with $p$ replaced by $ap$ and
$2C=|p|^2/2+c$. Let us verify the hypothesis of Lemma~\ref{lem:convex},
$$
   a|p| < C^2 = \left(\frac{|p|^2}{4}+\frac{c}{2}\right)^2.
$$
Using $c>\frac{3}{2}a^{2/3}$ and setting $x:=|p|$, it suffices to
show
$$
   16 a x \leq (x^2+3a^{2/3})^2 = x^4+6x^2a^{2/3}+9a^{4/3}.
$$
Replacing $x$ by $xa^{1/3}$, this is equivalent to non-negativity
of the function
$$
   g(x) = x^4+6x^2-16x+9.
$$
Since $g$ attains its minimum $0$ at $x=1$, the hypotheses of the
lemma are satisfied. Hence the level set $\Sigma$ is convex in $q$
for each fixed $p$. Adding $p=\infty$, $\Sigma$ gives a fiberwise
convex hypersurface in $T^*S^2$, where the coordinate on $S^2$ is
$p$ and the fiber coordinate is $q$. Let $F^*:T^*S^2\to\R$ be the
fundamental function of the corresponding Cartan space, which we
computed above to be
$$
   F^*(p,q) = \frac{1}{4}(|p|^2+2c)|q| \left(1 + \sqrt{1+\frac{16a\la
         p^\perp,q\ra}{|q|(|p|^2+2c)^2}} \right). 
$$
The fiberwise Legendre transform of $F^*$ (via $\frac{1}{2}{F^*}^2$)
yields a Finsler metric $F:TS^2\to\R$ whose geodesic flow equals the
Hamiltonian flow of $H$ on the energy level $\Sigma$.

Note that $H$, and thus also $F$, is invariant under the rotation of
$S^2$ around the vertical axis.

\section{Polar coordinates}
\label{sec:polar_coor}
To take advantage of the symmetry of $F^*$ it is useful to rewrite it
in polar coordinates $(x,y)$ defined by
$$p_1=x \cos y, \quad p_2=x \sin y.$$
The dual coordinates $(r,t)$ on the cotangent bundle are determined by
$$r dx+tdy=q_1 dp_1+q_2 dp_2.$$
In particular,
$$q_1=(\cos y)r-\frac{\sin y}{x}t, \quad
q_2=(\sin y)r+\frac{\cos y}{x}t$$ and hence
$$
   |q|^2=r^2+\frac{1}{x^2}t^2, \quad
   \langle p^\perp, q\rangle=-t,\quad
   |p|^2 = x^2.
$$ 
Hence in polar coordinates the fundamental function reads
$$
F^*(x,y,r,t)=\frac{1}{4}(x^2+2c)\sqrt{r^2+\frac{1}{x^2}t^2}
\Bigg(1+\sqrt{1-\frac{16at}{\sqrt{r^2+\frac{1}{x^2}t^2}(x^2+2c)^2}}
\Bigg).$$ 
Note that $F^*$ does not explicitly depend on $y$.

\subsection{A one-parameter family of Cartan metrics}
The above formula for the Cartan metrics $F^*$ describes a
two-parameter family with parameters $c$ and $a$. 
However, each such Cartan metric is a constant multiple of one with
$a$-parameter equal to $1$. 

\begin{proposition}
$F^*_{c,a}=a^{1/3}F^*_{ca^{-2/3},1}$.
\end{proposition}

\begin{proof}
Use the coordinate transformation $x=a^{1/3}x'$, $y=y'$.
On the cotangent fibers, this coordinate transformation induces
$r=a^{-1/3}r'$ and $t=t'$. We get
\begin{align*}
   &F^*_{c,a}(x,y,r,t) \cr
   &= \frac{1}{4}(a^{2/3}{x'}^2+2c)\sqrt{({r'}^2a^{-2/3} +
     \frac{1}{{x'}^2a^{2/3}}t'^2} \Bigg(1+\sqrt{1-\frac{16at'}
     {\sqrt{{r'}^2a^{-2/3} + \frac{1}{{x'}^2a^{2/3}}t'^2}(a^{2/3}
       x'^2+2c)^2}} \Bigg) \cr
   &= a^{1/3} \frac{1}{4}({x'}^2+2ca^{-2/3})\sqrt{{r'}^2 +
  \frac{1}{{x'}^2}t'^2} \Bigg(1+\sqrt{1-\frac{16t'}
  {\sqrt{{r'}^2+\frac{1}{{x'}^2}t'^2}( x'^2+2ca^{-2/3})^2}} \Bigg) \cr 
   &= a^{1/3} F^*_{ca^{-2/3},1}(x',y',r',t').
\end{align*}
\end{proof}
We shall henceforth put $a=1$.

\section{Spray coefficients}

An elegant way of computing the flag curvature of a Finsler surface
is via spray coefficients, see \cite[Section 12.5]{BCS}. Since we
are on the Cartan side we first have to Legendre transform the
fundamental function $F^*$, in order to apply the formulas in
\cite{BCS} directly. However, we do not know an explicit description
of the Legendre transform of $F^*$. 
We therefore derive in this section a formula to obtain the spray
coefficients of a Cartan 
surface directly from the fundamental function.
This allows us in turn to implement a MAPLE algorithm for computing
the flag curvature of a Cartan surface. 
We abbreviate
$$\mathcal{L}^*=\frac{1}{2}{F^*}^2$$
and denote by $\mathcal{L}$ the Legendre dual of $\mathcal{L}^*$.
The coordinates on the tangent bundle are given by
$$u(x,y,r,t)=\mathcal{L}^*_r(x,y,r,t)$$
$$v(x,y,r,t)=\mathcal{L}^*_t(x,y,r,t).$$
Here a subscript denotes differentiation with respect to the
corresponding variable. We first recall some useful facts of the
Legendre transformation for which we refer to \cite[Section 14.8]{BCS} 
or \cite{MHSS}. Since $\mathcal{L}^*$ and $\mathcal{L}$ are
2-homogeneous they numerically coincide
\begin{equation}\label{LL}
\mathcal{L}^*\Big(x,y,r,t\Big)=\mathcal{L}
\Big(x,y,u(x,y,r,t),v(x,y,r,t)\Big).
\end{equation}
 The cometric coefficients are
defined by
\begin{equation}
\label{eq:cometric}
g^{11}=\mathcal{L}^*_{rr},
\quad g^{12}=g^{21}=\mathcal{L}^*_{tr}, \quad
g^{22}=\mathcal{L}^*_{tt}.
\end{equation}
The metric
coefficients inverse to the cometric coefficients can be obtained by
$$g_{11}=\mathcal{L}_{uu},
\quad g_{12}=g_{21}=\mathcal{L}_{uv}, \quad
g_{22}=\mathcal{L}_{vv}.$$ Moreover, since the Legendre transform is
involutive, one has the identities
$$r=\mathcal{L}_u, \quad t=\mathcal{L}_v.$$
We further abbreviate the determinant of the cometric by
$$g=g^{11}g^{22}-\big(g^{12}\big)^2.$$
Differentiating (\ref{LL}) we
obtain
$$\frac{\partial \mathcal{L}^*}{\partial x}
=\frac{d\mathcal{L}}{dx}= \frac{\partial \mathcal{L}}{\partial x}
+\frac{\partial \mathcal{L}}{\partial u}\frac{\partial u}{\partial
x}+\frac{\partial \mathcal{L}}{\partial v}\frac{\partial v}{\partial
x}$$ 
and therefore
$$
\mathcal{L}_x=\mathcal{L}^*_x-\mathcal{L}_u \mathcal{L}^*_{r x}
-\mathcal{L}_v \mathcal{L}^*_{t x}
.
$$
Similarly
$$
\mathcal{L}_y=\mathcal{L}^*_y-\mathcal{L}_u \mathcal{L}^*_{r y}
-\mathcal{L}_v \mathcal{L}^*_{t y}
.
$$ 
From the identity
$$
0=\frac{\partial r}{\partial x}=
\frac{\partial \mathcal{L}_u}{\partial x}+ \frac{\partial
\mathcal{L}_u}{\partial u}\frac{\partial u}{\partial x}
+\frac{\partial \mathcal{L}_u}{\partial v}\frac{\partial v}{\partial
x}
$$ 
one obtains
\begin{eqnarray*} \mathcal{L}_{ux} &=& -\mathcal{L}_{uu}
\mathcal{L}^*_{r x}-\mathcal{L}_{uv}\mathcal{L}^*_{t x}\\
&=&-g_{11}\mathcal{L}^*_{r x}-g_{12} \mathcal{L}^*_{t x}\\
&=&\frac{1}{g}\Big(-g^{22}\mathcal{L}^*_{r x}+g^{12}\mathcal{L}^*_{t
x}\Big)
\end{eqnarray*}
In a similar vein, one derives the formulas
\begin{eqnarray*}
\mathcal{L}_{uy}&=&-\mathcal{L}_{uu}\mathcal{L}^*_{r y}
-\mathcal{L}_{uv} \mathcal{L}^*_{t y}\\
&=&\frac{1}{g}\Big(-g^{22}\mathcal{L}^*_{r y}+g^{12}
\mathcal{L}^*_{t y}\Big)
\end{eqnarray*}
\begin{eqnarray*}
\mathcal{L}_{vx}&=&-\mathcal{L}_{vu}\mathcal{L}^*_{r
x}-\mathcal{L}_{vv}\mathcal{L}^*_{t x}\\
&=&\frac{1}{g}\Big(g^{12}\mathcal{L}^*_{r
x}-g^{11}\mathcal{L}^*_{tx}\Big)
\end{eqnarray*}
\begin{eqnarray*}
\mathcal{L}_{v y}&=&-\mathcal{L}_{vu}\mathcal{L}^*_{r y}
-\mathcal{L}_{vv}\mathcal{L}^*_{t y}\\
&=&\frac{1}{g}\Big(g^{12}\mathcal{L}^*_{r y}-g^{11}\mathcal{L}^*_{t
y}\Big).
\end{eqnarray*}
Plugging these identities into the formulas for the spray coefficients
$G,H$ (this $H$ is not the Hamiltonian!)  
in \cite[p 330]{BCS} we get
\begin{eqnarray*}
2G&=&g\mathcal{L}_{vv}\mathcal{L}_x-g\mathcal{L}_{vx}\mathcal{L}_v
-g\mathcal{L}_{uv}\mathcal{L}_y+g\mathcal{L}_{uy}\mathcal{L}_v\\
&=&g^{11}\Big(\mathcal{L}^*_x-\mathcal{L}_u\mathcal{L}^*_{r
x}-\mathcal{L}_v \mathcal{L}^*_{t
x}\Big)-\Big(g^{12}\mathcal{L}^*_{r
x}-g^{11}\mathcal{L}^*_{tx}\Big)\mathcal{L}_v\\
& &+g^{12}\Big(\mathcal{L}^*_y-\mathcal{L}_u \mathcal{L}^*_{r y}
-\mathcal{L}_v \mathcal{L}^*_{t y}\Big)+
\Big(-g^{22}\mathcal{L}^*_{r y}+g^{12} \mathcal{L}^*_{t y}\Big)
\mathcal{L}_v\\
&=&\Big(g^{11}\mathcal{L}^*_x+g^{12}\mathcal{L}^*_y\Big) -\Big(
g^{11}\mathcal{L}^*_{r x}+g^{12}\mathcal{L}^*_{r y}\Big)
\mathcal{L}_u\\
& &+\Big(g^{11}\mathcal{L}^*_{t x}-g^{12}\mathcal{L}^*_{r x}
-g^{11}\mathcal{L}^*_{t x}-g^{12}\mathcal{L}^*_{t y}
-g^{22}\mathcal{L}^*_{r y}+g^{12}\mathcal{L}^*_{t y}\Big)
\mathcal{L}_v\\
&=&\Big(g^{11}\mathcal{L}^*_x+g^{12}\mathcal{L}^*_y\Big) -\Big(
g^{11}\mathcal{L}^*_{r x}+g^{12}\mathcal{L}^*_{r y}\Big)
\mathcal{L}_u-\Big(g^{12}\mathcal{L}^*_{r x} +g^{22}\mathcal{L}^*_{r
y}\Big) \mathcal{L}_v
\end{eqnarray*}
\begin{eqnarray*}
2H&=&g\mathcal{L}_{uu}\mathcal{L}_y-g\mathcal{L}_{uy}\mathcal{L}_u
-g\mathcal{L}_{vu}\mathcal{L}_x+g\mathcal{L}_{vx}\mathcal{L}_u\\
&=&g^{22}\Big(\mathcal{L}^*_y-\mathcal{L}_u \mathcal{L}^*_{r y}
-\mathcal{L}_v \mathcal{L}^*_{t y}\Big)
-\Big(-g^{22}\mathcal{L}^*_{r y}+g^{12} \mathcal{L}^*_{t
y}\Big)\mathcal{L}_u\\
& &+g^{12}\Big(\mathcal{L}^*_x-\mathcal{L}_u \mathcal{L}^*_{r x}
-\mathcal{L}_v \mathcal{L}^*_{t x}\Big)+ \Big(g^{12}\mathcal{L}^*_{r
x}-g^{11}\mathcal{L}^*_{tx}\Big)
\mathcal{L}_u\\
&=&\Big(g^{12}\mathcal{L}^*_x+g^{22}\mathcal{L}^*_y\Big)
-\Big(g^{12}\mathcal{L}^*_{t x}+g^{22}\mathcal{L}^*_{t y}\Big)
\mathcal{L}_v\\
& &+\Big(g^{22}\mathcal{L}^*_{r y}-g^{12}\mathcal{L}^*_{t y}
-g^{22}\mathcal{L}^*_{r y}-g^{12}\mathcal{L}^*_{r x} +g^{12}
\mathcal{L}^*_{r x}-g^{11}\mathcal{L}^*_{t
x}\Big)\mathcal{L}_u\\
&=&\Big(g^{12}\mathcal{L}^*_x+g^{22}\mathcal{L}^*_y\Big)
-\Big(g^{12}\mathcal{L}^*_{t x}+g^{22}\mathcal{L}^*_{t y}\Big)
\mathcal{L}_v-\Big(g^{11}\mathcal{L}^*_{t x}+g^{12}\mathcal{L}^*_{t
y}\Big)\mathcal{L}_u
\end{eqnarray*}

\subsection{A formula for the flag curvature following Bao-Chern-Shen}
We shall use formula (12.5.14) from \cite{BCS} to compute the flag
curvature. This formula is given by
\begin{equation}
\label{eq:flag_curvature}
K=\frac{(G_{xv}-G_{yu})v+2GG_{uu}+2HG_{uv}-G_uG_u-G_vH_u}
{v \mathcal L_v}.
\end{equation}
All these terms can be rewritten entirely into terms that exist on the Cartan side.
In particular, we shall use the above formulas for $G$ and $H$ in
terms of $g^{ij}$ and $\mathcal L^*$ and their derivatives. 

As before, we define the Hamiltonian or Legendre dual of $\mathcal L$ by
$$
   {\mathcal L}^*=\frac{1}{2}{F^*}^2
$$
The cometric can be obtained via Equation~\eqref{eq:cometric}.
As the Hamiltonian is given in terms of the base coordinates $x,y$ and cotangent fiber coordinates $r,t$, we need to apply the chain rule to compute \eqref{eq:flag_curvature}, which is given in terms of $x,y$ and the tangent fiber coordinates $u,v$.
We need to use the following relations,
$$
\frac{\partial}{\partial u}=
\frac{\partial r}{\partial u} \frac{\partial}{\partial r }
+\frac{\partial t}{\partial u} \frac{\partial}{\partial t }=
g_{11} \frac{\partial}{\partial r }
+g_{21} \frac{\partial}{\partial t }.
$$
Similarly,
$$
\frac{\partial}{\partial v}=
g_{12} \frac{\partial}{\partial r }
+g_{22} \frac{\partial}{\partial t }.
$$
Note that the metric can be expressed in terms of the cometric and its determinant.
The dependence on the base coordinate is a little tricky, since the independent variables are now $x,y,u,v$, whereas the Hamiltonian is given in terms $x,y,r,t$.
This means that the (co)-metric comes in.
Let us write out the required expression.
We shall use the notation
$$
\left( \frac{\partial G}{\partial x} \right)_{y,u,v}
$$
to indicate the derivative of $G$ keeping $y,u,v$ fixed.
We obtain
\begin{eqnarray*}
\left( \frac{\partial G}{\partial x} \right)_{y,u,v}
=&&
\left( \frac{\partial G}{\partial x} \right)_{y,r,t}
+
\left( \frac{\partial G}{\partial r} \right)_{x,y,t}
\left( \frac{\partial r}{\partial x} \right)_{y,u,v}
+
\left( \frac{\partial G}{\partial t} \right)_{x,y,r}
\left( \frac{\partial t}{\partial x} \right)_{y,u,v}\\
=&&
\left( \frac{\partial G}{\partial x} \right)_{y,r,t}
+
\left( \frac{\partial G}{\partial r} \right)_{x,y,t}
\left(
\left( \frac{\partial g_{11}}{\partial x} \right)_{y,u,v}u
+
\left( \frac{\partial g_{12}}{\partial x} \right)_{y,u,v}v
\right)
\\
&&
+
\left( \frac{\partial G}{\partial t} \right)_{x,y,r}
\left(
\left( \frac{\partial g_{21}}{\partial x} \right)_{y,u,v}u
+
\left( \frac{\partial g_{22}}{\partial x} \right)_{y,u,v}v
\right)
.
\end{eqnarray*}
Similarly, we have
\begin{eqnarray*}
\left( \frac{\partial G}{\partial y} \right)_{x,u,v}
=&&
\left( \frac{\partial G}{\partial y} \right)_{x,r,t}
+
\left( \frac{\partial G}{\partial r} \right)_{x,y,t}
\left(
\left( \frac{\partial g_{11}}{\partial y} \right)_{x,u,v}u
+
\left( \frac{\partial g_{12}}{\partial y} \right)_{x,u,v}v
\right)
\\
&&
+
\left( \frac{\partial G}{\partial t} \right)_{x,y,r}
\left(
\left( \frac{\partial g_{21}}{\partial y} \right)_{x,u,v}u
+
\left( \frac{\partial g_{22}}{\partial y} \right)_{x,u,v}v
\right)
.
\end{eqnarray*}
In the above we have used
$$
r=g_{11}u+g_{12}v \text{ and }t=g_{21}u+g_{22}v.
$$
Using the cometric $g^{ij}$, we can conversely express $u$ and $v$ in terms of $r$ and $t$.
Hence we can translate all derivatives in the expression for the flag curvature in terms of quantities on the Cartan side.

\section{Computing the Flag curvature of the rotating Kepler problem}
We now come to the computer implementation of the curvature computation.
The MAPLE-program is included in the appendix.
Although the program itself is straightforward with the above formulas, there are some practical issues involved which we shall briefly discuss.

Even though our computer setup was sufficient to compute the full expression for the flag curvature of the Cartan metric $F^*_{c,1}$ as a function of $x,r,t$, the resulting expression was too unwieldy to simplify to reasonable proportions.
However, it is possible to evaluate and simplify the curvature exactly in specific points, for instance if the $r$ coordinate is $0$.
The simplified expression for the curvature along this ray is still very long, see Equation~\eqref{eq:flag_curvature_rotating_kepler}.

In order to make these evaluations work, one only needs to realize that it is beneficial for a computer to quickly evaluate the results as soon as further differentiation is not needed.
This is the reason for the many substitutions in the MAPLE program.

\subsection{Analytical and numerical results}
We first fix the fiber coordinate $(r,t)=(0,x)$ and compute the curvature as a function of $x$, which is the coordinate on the base sphere $S^2$.
The MAPLE program yields the following expression for the flag curvature.
Write 
$$
\alpha:=\sqrt {{x}^{4}+4\,{x}^{2}c+4\,{c}^{2}-16\,
x}
$$
to simplify the resulting expression.
\begin{equation}
\label{eq:flag_curvature_rotating_kepler}
K(F^*_{c,1})(x,0,0,x)=\frac{2}{ \left( {x}^{2}+2\,c+\alpha \right)  \left( {x}^{2}\alpha+2\,c\alpha+{x}^{4}+4\,{x}^{2}
c+4\,{c}^{2}-8\,x \right)  \left( {x}^{4}+4\,{x}^{2}c+4\,{c}^{2}-16\,x
 \right) ^{2}}\cdot
\end{equation}
\begin{eqnarray*}
&&
\left(
5824\,{x}^{2}{c}^{4}-5888\,{x}^{3}{c}^{5}-3840\,{x}^{2}c-
2240\,{c}^{6}x-6320\,{x}^{5}{c}^{4}-384\,{x}^{2}\alpha+1120\,{x}^{6}{c}^{5}
+2\,{x}^{14}c
\right.
\\
&&
+28\,{x}^{12}{
c}^{2}-6528\,{x}^{5}c+256\,{c}^{8}-864\,\alpha{c}^{5}x-1872\,\alpha{c}^{4}{x}^{3}
+896\,{x}^{2}{c}^{7} -1296\,{x}^{7}+204\,{x}^{10}
\\
&&
-768
\,{c}^{5}-9\,{x}^{13}+2096\,{x}^{8}c-160\,{x}^{11}c-1060\,{c}^{2}{x}^{
9}-3520\,{c}^{3}{x}^{7}+11520\,{x}^{4}{c}^{3}
+3840\,{c}^{3}x
\\
&&
+7584\,{x}
^{6}{c}^{2}-5952\,{x}^{3}{c}^{2}-648\,{x}^{7}\alpha{c}^{2}-126\,{x}^{9}\alpha c
-1120\,{x}^{3}\alpha c+2448\,{x}^{4}\alpha{c}^{2}+
1152\,\alpha{c}^{2}x
\\
&&
+1920\,{x}^{
4}+168\,{x}^{10}{c}^{3}+1344\,{x}^{4}{c}^{6}+560\,{x}^{8}{c}^{4}+1032
\,{x}^{6}\alpha c-1584\,{x}^{5}
\alpha{c}^{3}
+1632\,\alpha{c}^{3}{x}^{2}
\\
&&
+128\,\alpha{c}^{7}
+384\,\alpha{c}^{6}{x}^{2}
-9\,{x}^{11}\alpha+2\,{x}^{12}\alpha c
+132\,{x}^{8}\alpha
+320\,{x}^{6}
\alpha{c}^{4}+120\,{x}^{8}\alpha{c}^{3}
\\
&&
\left.+24\,{x}^{10}\alpha{c}^{2}
+480\,\alpha{c}^{5}{x}^{4}-528\,\alpha{x}^{5}-384\,\alpha
{c}^{4}
\right)
\end{eqnarray*}
Since this expression is too complicated to deduce anything meaningful
by hand, we have plotted some graphs in 
Figures~\ref{fig:K=1.51},~\ref{fig:K=1.65} and~\ref{fig:K=2}. 
We see that if $c$ is small enough, then the flag curvature can become
negative. Note that the graphs are also plotted for $x<0$, where the
sign of $t$ changes since we insert the covector $(r,t)=(0,x)$.  

\begin{figure}[htp]
\centering
\includegraphics[width=0.8\textwidth,clip]{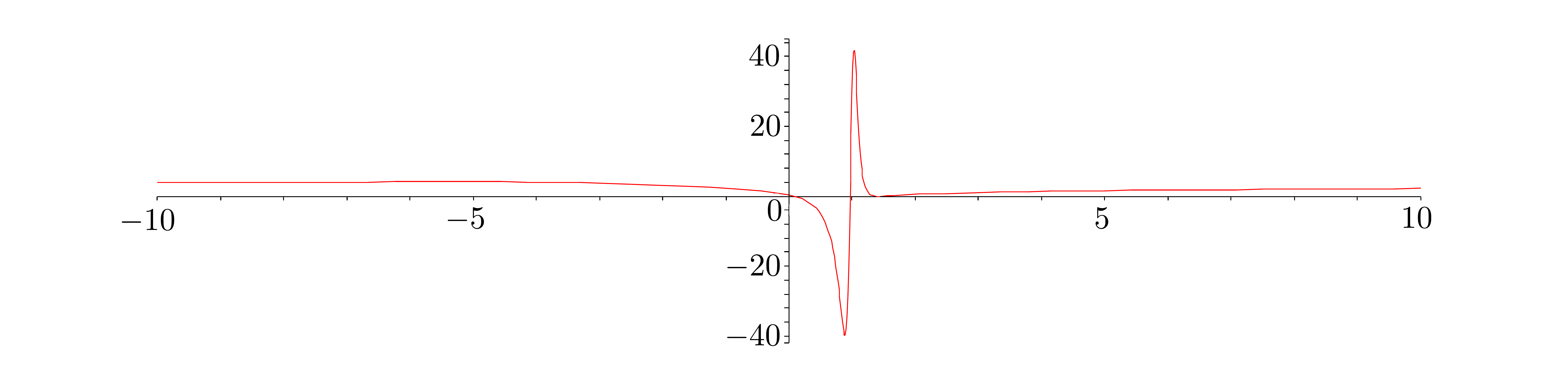}%
\caption{Flag curvature as a function of $x$ for $r=0$, $t=\pm 1$ and $c=1.51$}
\label{fig:K=1.51}
\end{figure}
\begin{figure}[htp]
\centering
\includegraphics[width=0.8\textwidth,clip]{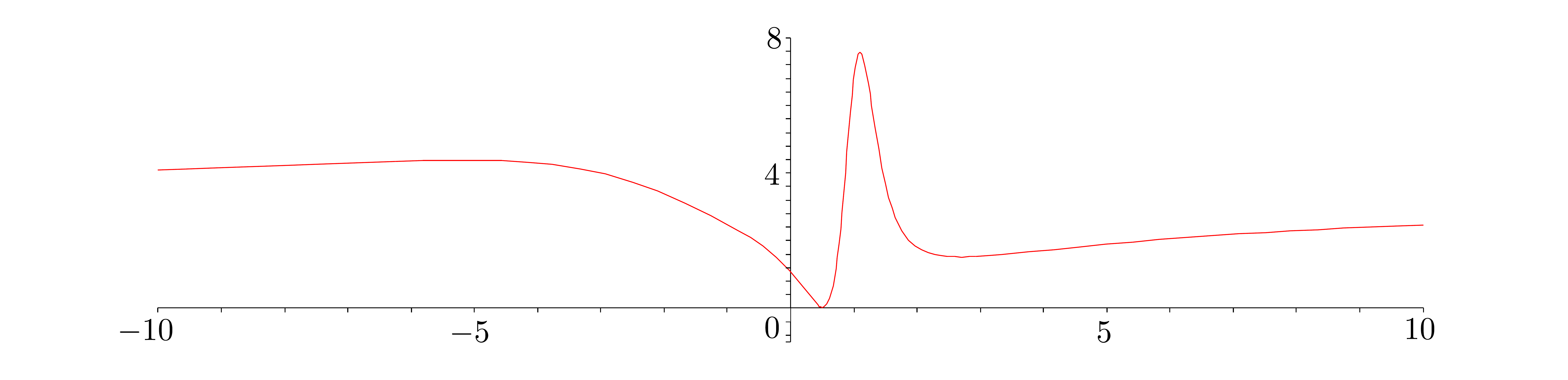}%
\caption{Flag curvature as a function of $x$ for $r=0$, $t=\pm 1$ and $c=1.65$}
\label{fig:K=1.65}
\end{figure}
\begin{figure}[htp]
\centering
\includegraphics[width=0.8\textwidth,clip]{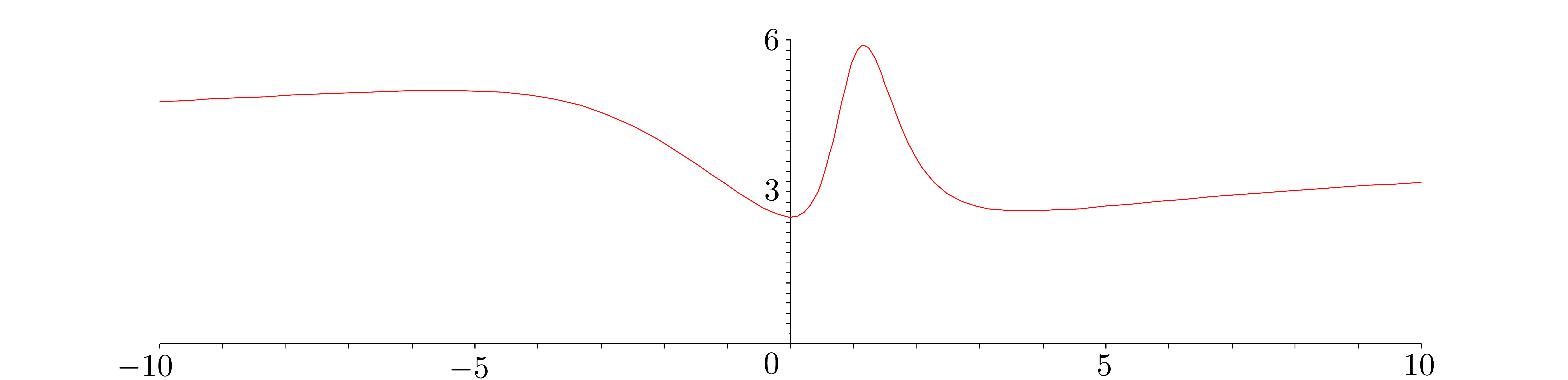}%
\caption{Flag curvature as a function of $x$ for $r=0$, $t=\pm 1$ and $c=2$}
\label{fig:K=2}
\end{figure}

Let us now consider the flag curvature of $K_c$ as a function on the
unit cotangent bundle. We shall use the coordinates $x,\phi$, where
the cotangent fiber coordinates are given by 
$$
   (r,t)=(\sin \phi,\cos \phi ).
$$
The full expression for the curvature for $r\neq 0$ is considerably
more complicated than
Equation~\eqref{eq:flag_curvature_rotating_kepler}, so 
we shall only present numerical results for this general case.

Using MAPLE we made a graph of the flag curvature by evaluating the
curvature on a $256 \times 256$ lattice, see
Figure~\ref{figure:graph_flag_curvature}.  
On this lattice the minimal value of the flag curvature for $c=1.55$ was
$-5.55$, and the maximal value $15.21$.

\begin{figure}[htp]
\centering
\includegraphics[width=0.7\textwidth,clip]{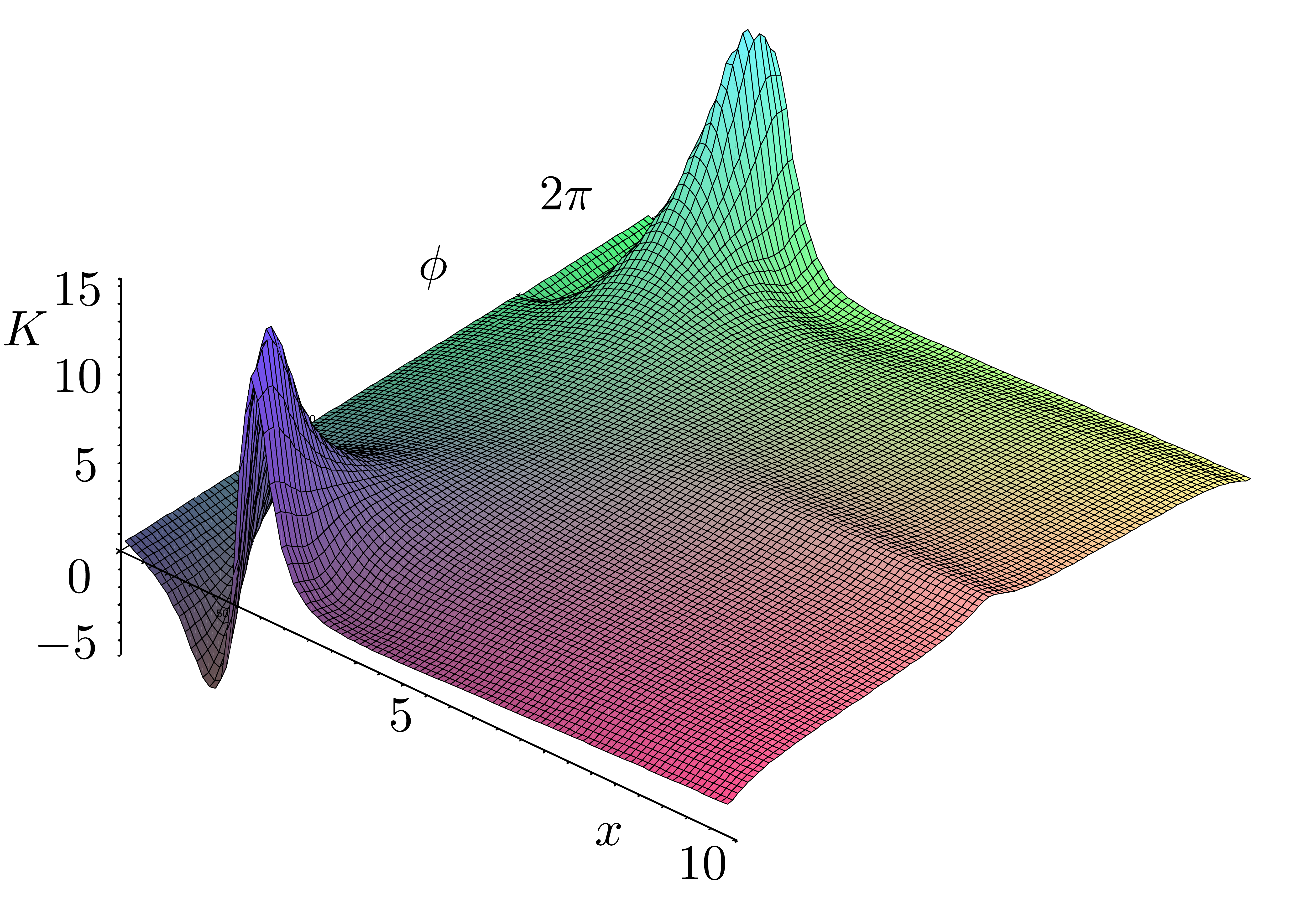}%
\caption{Flag curvature as a function of $x$ and $\phi$ for $c=1.55$}
\label{figure:graph_flag_curvature}
\end{figure}

\section{Appendix}
Here is a general program to (numerically) compute the flag curvature of a Cartan metric.
\begin{verbatim}
> K:=proc(F,x0,y0,r0,t0)
> local L,Lx,Ly,Lr,Lt,Lxt,Lyr,Lrr,Lrt,Ltt,D,
>        G,H,Gr,Gt,Gar,Gat,Grr,Grt,Gxt,Gyr,Hr,K,
>        Garre,Garte,Grre,Grte,Gxte,Gyre,Hre,Gre,Gte,
>        Ht,Hrr,Hrt,Htt,detg,gi,g,gie,ge,
>        Lsx,Lsy,Lsq1x,Lsq1y,Lsq2x,Lsq2y:
> 
> ## Hamiltonian given
> H:=F^2/2:
> 
> Hr:=diff(H,r):
> Ht:=diff(H,t):
> Hrr:=diff(Hr,r):
> Hrt:=diff(Hr,t):
> Htt:=diff(Ht,t):
> 
> gi[1,1]:=Hrr: gi[1,2]:=Hrt: gi[2,1]:=Hrt: gi[2,2]:=Htt:  
> detg:=gi[1,1]*gi[2,2]-gi[1,2]^2:
> 
> ##obtain Finsler 'metric' g_{ij}: note that the dependence on the fiber 
> ##coordinates is not correct, since we haven't performed a Legendre 
> ##transformation
> g[1,1]:=1/detg*gi[2,2]: g[1,2]:=-1/detg*gi[1,2]: 
> g[2,1]:=-1/detg*gi[2,1]: g[2,2]:=1/detg*gi[1,1]:
> 
> Lr:=r:
> Lt:=t:
> 
> Lsx:=diff(H,x): Lsy:=diff(H,y):
> Lsq1x:=diff(Hr,x): Lsq1y:=diff(Hr,y):
> Lsq2x:=diff(Ht,x): Lsq2y:=diff(Ht,y):
> 
> ##Formulas from Section 5
> G:= 1/2*( gi[1,1]*Lsx+gi[1,2]*Lsy-Lr*(gi[1,1]*Lsq1x+gi[1,2]*Lsq1y)
>           -Lt*(gi[1,2]*Lsq1x+gi[2,2]*Lsq1y) ):
> H:=1/2*(gi[1,2]*Lsx+gi[2,2]*Lsy-Lt*(gi[1,2]*Lsq2x+gi[2,2]*Lsq2y)
>           -Lr*(gi[1,1]*Lsq2x+gi[1,2]*Lsq2y) ):
> 
> ##avoid double computations
> Gar:=diff(G,r): Gat:=diff(G,t):
> 
> Gr:=Gar*g[1,1]+Gat*g[2,1]:
> Gt:=Gar*g[1,2]+Gat*g[2,2]:
> 
> ##evaluate quickly
> 
> ge[1,1]:=subs(x=x0,y=y0,r=r0,t=t0,g[1,1]): 
> ge[1,2]:=subs(x=x0,y=y0,r=r0,t=t0,g[1,2]): 
> ge[2,1]:=subs(x=x0,y=y0,r=r0,t=t0,g[2,1]): 
> ge[2,2]:=subs(x=x0,y=y0,r=r0,t=t0,g[2,2]):
> gie[1,1]:=subs(x=x0,y=y0,r=r0,t=t0,gi[1,1]): 
> gie[1,2]:=subs(x=x0,y=y0,r=r0,t=t0,gi[1,2]): 
> gie[2,1]:=subs(x=x0,y=y0,r=r0,t=t0,gi[2,1]): 
> gie[2,2]:=subs(x=x0,y=y0,r=r0,t=t0,gi[2,2]):
> 
> ##avoid double computations
> Garre:=subs(x=x0,y=y0,r=r0,t=t0,diff(Gr,r) ): 
> Garte:=subs(x=x0,y=y0,r=r0,t=t0, diff(Gr,t) ): 
> 
> ##evaluate when no more differentation is needed.
> Grre:=Garre*ge[1,1]+Garte*ge[2,1] :
> Grte:=Garre*ge[1,2]+Garte*ge[2,2] :
> 
> 
> Gxte:=subs(x=x0,y=y0,r=r0,t=t0, diff(Gt,x) )
>       +subs(x=x0,y=y0,r=r0,t=t0, diff(Gt,r) )
>        *subs(x=x0,y=y0,r=r0,t=t0,( diff(g[1,1],x)
>        *(gie[1,1]*r+gie[1,2]*t)+diff(g[1,2],x)*(gie[2,1]*r+gie[2,2]*t) ) ) 
>       +subs(x=x0,y=y0,r=r0,t=t0, diff(Gt,t) )
>        *subs(x=x0,y=y0,r=r0,t=t0,( diff(g[2,1],x)
>        *(gie[1,1]*r+gie[1,2]*t)+diff(g[2,2],x)*(gie[2,1]*r+gie[2,2]*t) ) ):
> 
> 
> Gyre:=subs(x=x0,y=y0,r=r0,t=t0,diff(Gr,y) )
>       +subs(x=x0,y=y0,r=r0,t=t0,diff(Gr,r) )
>        *subs(x=x0,y=y0,r=r0,t=t0,( diff(g[1,1],y)
>        *(gie[1,1]*r+gie[1,2]*t)+diff(g[1,2],y)*(gie[2,1]*r+gie[2,2]*t) ) ) 
>       +subs(x=x0,y=y0,r=r0,t=t0,diff(Gr,t) )
>        *subs(x=x0,y=y0,r=r0,t=t0,( diff(g[2,1],y)
>        *(gie[1,1]*r+gie[1,2]*t)+diff(g[2,2],y)*(gie[2,1]*r+gie[2,2]*t) ) ):
> 
> 
> Hre:=subs(x=x0,y=y0,r=r0,t=t0,diff(H,r) )*ge[1,1]
>       +subs(x=x0,y=y0,r=r0,t=t0,diff(H,t) )*ge[2,1] :
> 
> 
> Gre:=subs(x=x0,y=y0,r=r0,t=t0,Gr):
> Gte:=subs(x=x0,y=y0,r=r0,t=t0,Gt):
> 
> K:= subs(x=x0,y=y0,r=r0,t=t0, 
>       ( (Gxte-Gyre)*(gie[2,1]*r+gie[2,2]*t)+2*G*Grre+2*H*Grte-Gre*Gre-Gte*Hre )
>       /( (gie[2,1]*r+gie[2,2]*t)*Lt)  ):
> 
> 
> end proc;

> a:=1:c:=1.55:x0:=0.75: y0:=0: r0:=0: t0:=1:
> F:=1/4*(x^2+2*c)*sqrt(r^2+x^(-2)*t^2)*
           ( 1+1*sqrt(1-16*a*(t)/sqrt(r^2+x^(-2)*t^2)/(x^2+2*c)^2 )  );

> K(F,x0,y0,r0,t0);
\end{verbatim}
The following program is optimized for the rotationally symmetric Cartan metric associated with the rotating Kepler Hamiltonian.
Here we use the same notation as in Section~\ref{sec:polar_coor}.
\begin{verbatim}
> K:=proc(F)
> local H,Hr,Ht,Hrr,Hrt,Htt,gi,detg,g,
>        Lr,Lt,Lsx,Lsq1x,Lsq2x,
>        G,Gr,Gt,Gar,Gat,Grre,Gtt,Gre,Gte,Hre,x0,r0,t0,
>        Garre,Garte,Grte,Gxte,
>        gie,detge,ge,Lsq1xe,Lsq2xe:
> 
> ## Hamiltonian given
> H:=F^2/2:
> 
> Hr:=diff(H,r):
> Ht:=diff(H,t):
> Hrr:=diff(Hr,r):
> Hrt:=diff(Hr,t):
> Htt:=diff(Ht,t):
> 
> gi[1,1]:=Hrr: gi[1,2]:=Hrt: gi[2,1]:=Hrt: gi[2,2]:=Htt:  
> 
> detg:=gi[1,1]*gi[2,2]-gi[1,2]^2:
> 
> ##obtain Finsler 'metric' g_{ij}: note that the dependence on the fiber 
> ##coordinates is not correct, since we haven't performed a Legendre 
> ##transformation
> g[1,1]:=1/detg*gi[2,2]: g[1,2]:=-1/detg*gi[1,2]: 
> g[2,1]:=-1/detg*gi[2,1]: g[2,2]:=1/detg*gi[1,1]:
> 
> 
> Lr:=r:
> Lt:=t:
> 
> Lsx:=diff(H,x): 
> Lsq1x:=diff(Hr,x): 
> Lsq2x:=diff(Ht,x): 
> 
> ##Formulas from Section 5
> G:= 1/2*( gi[1,1]*Lsx-Lr*(gi[1,1]*Lsq1x)-Lt*(gi[1,2]*Lsq1x) ):
> H:=1/2*(gi[1,2]*Lsx-Lt*(gi[1,2]*Lsq2x)-Lr*(gi[1,1]*Lsq2x) ):
> 
> ##avoid double computations
> Gar:=diff(G,r): Gat:=diff(G,t):
> 
> Gr:=Gar*g[1,1]+Gat*g[2,1]:
> Gt:=Gar*g[1,2]+Gat*g[2,2]:
> 
> ##evaluate quickly
> x0:=x: r0:=0: t0:=x:
> ge[1,1]:=subs(x=x0,r=r0,t=t0,g[1,1]): ge[1,2]:=subs(x=x0,r=r0,t=t0,g[1,2]): 
> ge[2,1]:=subs(x=x0,r=r0,t=t0,g[2,1]): ge[2,2]:=subs(x=x0,r=r0,t=t0,g[2,2]):
> gie[1,1]:=subs(x=x0,r=r0,t=t0,gi[1,1]): gie[1,2]:=subs(x=x0,r=r0,t=t0,gi[1,2]): 
> gie[2,1]:=subs(x=x0,r=r0,t=t0,gi[2,1]): gie[2,2]:=subs(x=x0,r=r0,t=t0,gi[2,2]):
> detge:=subs(x=x0,r=r0,t=t0,detg):
> Lsq1xe:=subs(x=x0,r=r0,t=t0,Lsq1x):
> Lsq2xe:=subs(x=x0,r=r0,t=t0,Lsq2x):
> 
> ##avoid double computations
> Garre:=subs(x=x0,r=r0,t=t0,diff(Gr,r) ): 
> Garte:=subs(x=x0,r=r0,t=t0, diff(Gr,t) ): 
> 
> ##evaluate when no more differentation is needed.
> Grre:=Garre*ge[1,1]+Garte*ge[2,1] :
> Grte:=Garre*ge[1,2]+Garte*ge[2,2] :
> 
> Gxte:=subs(x=x0,r=r0,t=t0, diff(Gt,x) )
>         +subs(x=x0,r=r0,t=t0, diff(Gt,r) )*(-ge[1,1]*Lsq1xe-ge[1,2]*Lsq2xe) 
>         +subs(x=x0,r=r0,t=t0, diff(Gt,t) )*(-ge[1,2]*Lsq1xe-ge[2,2]*Lsq2xe):
> 
> Hre:=subs(x=x0,r=r0,t=t0,diff(H,r) )*ge[1,1]
>       +subs(x=x0,r=r0,t=t0,diff(H,t) )*ge[2,1] :
> 
> Gre:=subs(x=x0,r=r0,t=t0,Gr):
> Gte:=subs(x=x0,r=r0,t=t0,Gt):
> 
> K:= subs(x=x0,r=r0,t=t0, 
>      ( (Gxte)*(gie[2,1]*r+gie[2,2]*t)+2*G*Grre+2*H*Grte-Gre*Gre-Gte*Hre )
>      /( (gie[2,1]*r+gie[2,2]*t)*Lt)  ):
> 
> end proc;

> F:=1/4*(x^2+2*c)*sqrt(r^2+x^(-2)*t^2)*
           ( 1+1*sqrt(1-16*a*(t)/sqrt(r^2+x^(-2)*t^2)/(x^2+2*c)^2 )  );

> GK:=K(F);
> curv:=simplify(GK,assume=positive);
> plot({subs(a=1,c=2,curv)},x=-10..10);
\end{verbatim}



\begin{thebibliography}{999}

\bibitem{AFKP} P.\,Albers, U.\,Frauenfelder, O.\,van Koert,
G.\,Paternain, {\em The contact geometry of the restricted 3-body
problem}, arXiv:1012.2140.
\bibitem{AFFHK} P.\,Albers, J.\,Fish, U.\,Frauenfelder, H.\,Hofer, O.\,van Koert, {\em Global surfaces of section in the planar restricted 3-body problem}, arXiv:1103.3881.
\bibitem{BCS} D.~Bao, S.-S.~Chern and Z.~Shen, {\em An Introduction to
    Riemann-Finsler Geometry}, Springer (2000).
\bibitem{belbruno1} E.\,Belbruno, \emph{Lunar capture orbits, a
method of constructing earth-moon trajectories and the lunar gas
mission}, Proceedings of AIAA/DGGLR/JSASS
Inter.\,Elec.\,Propl.\,Conf., nomber 87-1054 (1987).
\bibitem{belbruno-miller} E.\,Belbruno, J.\,Miller,
\emph{A ballistic lunar capture trajectory for the Japanese
spacecraft Hiten}, Technical Report JPL-IOM 312/90.4-1731-EAB, Jet
Propulsion Laboratory (1990).
\bibitem{conley1} C.\,Conley, \emph{Low energy transit orbits in the
restricted three-body problem}, SIAM J.\,Appl.\,Math. \textbf{16},
732--746 (1968).
\bibitem{conley2} C.\,Conley, \emph{Twist mappings, Analyticity
and Periodic Solutions which Pass Close to an Unstable Periodic
Solution}, Benjamin, New York, 129--154 (1968).
\bibitem{conley3} C.\,Conley, \emph{On the ultimate behavior of
orbits with respect to an unstable critical point I, oscillating,
asymptotic and capture orbits}, J.\,Differential Equations, 5:
136--158 (1969).
\bibitem{HP} C.\,Conley, \emph{Dynamically convex Finsler metrics and $J$J-holomorphic embedding of asymptotic cylinders}, Ann. Global Anal. Geom. 34 (2008), no. 2, 115–134. 
\bibitem{M} J.~Moser, {\em Regularization of Kepler's problem and the
averaging method on a manifold}, Comm.\,Pure Appl.\,Math.
\textbf{23} (1970) 609-636.
\bibitem{MHSS} R.\,Miron, D.\,Hrimiuc, H.\,Shimada, S.\,Sabau,
{\em The geometry of Hamilton and Lagrange spaces}, Fundamental
Theories of Physics, \textbf{118}, Kluwer Academic Publishers Group,
Dordrecht (2001).
\end{thebibliography}
\end{document}